\documentclass{article}
\usepackage{amsmath}
\usepackage{amssymb}
\usepackage{amsfonts}

\newtheorem{theorem}{Theorem}[section]

\newtheorem{example}[theorem]{Example}

\newtheorem{lemma}[theorem]{Lemma}

\newenvironment{proof}[1][Proof]{\textbf{#1.} }{\ \rule{0.5em}{0.5em}}
\begin{document}

\author{Miko\l aj Pep\l o\'{n}ski}
\title{On the existence and multiplicity of solutions for a fourth order
discrete BVP}
\maketitle

\begin{abstract}
We investigate the existence and multiplicity of solutions for fourth order
discrete boundary value problems via critical point theory.
\end{abstract}

\section{Introduction}

Difference equations have been applied as mathematical models in diverse
areas, such as finanse insurance, economy,disease control, biology, physics,
mechanics, computer science - see \cite{elyadi}. It is important to know the
conditions which guarantee the existence and muliplicity of solutions.

For fixed $a,b\in 
%TCIMACRO{\U{2124} }%
%BeginExpansion
\mathbb{Z}
%EndExpansion
$ we define $Z[a,b]=[a,b]\cap 
%TCIMACRO{\U{2124} }%
%BeginExpansion
\mathbb{Z}
%EndExpansion
$. and $\Delta $ is the forward difference operator 
\begin{equation*}
\Delta x(k)=x(k+1)-x(k).
\end{equation*}

In this note we consider a Dirichlet boundary value problem (briefly BVP)
for a fourth order discrete equation

\begin{equation}
\left\{ 
\begin{array}{l}
\Delta ^{2}(p(k)\Delta ^{2}y(k-2))+\Delta (q(k)\Delta y(k-1))+f(k,y(k))=0%
\text{ for }k\in Z[1,N]\bigskip \\ 
y(-1)=y(0)=y(N+1)=y(N+2)=0%
\end{array}%
\right.  \label{BVP}
\end{equation}

Where $N\geq 1$, $f:Z[1,N]\times 
%TCIMACRO{\U{211d} }%
%BeginExpansion
\mathbb{R}
%EndExpansion
\rightarrow 
%TCIMACRO{\U{211d} }%
%BeginExpansion
\mathbb{R}
%EndExpansion
$ is continous on its second variable for all $k\in Z[1,N]$, $%
p:Z[1,N+2]\rightarrow 
%TCIMACRO{\U{211d} }%
%BeginExpansion
\mathbb{R}
%EndExpansion
$, $q:Z[1,N+1]\rightarrow 
%TCIMACRO{\U{211d} }%
%BeginExpansion
\mathbb{R}
%EndExpansion
$. By a solution to a problem (\ref{BVP}) we mean such a function $%
y:Z[-1,N+2]\rightarrow 
%TCIMACRO{\U{211d} }%
%BeginExpansion
\mathbb{R}
%EndExpansion
$ which satisfies the difference equation and the given boundary conditions.
The main purpose of this paper is to study the multiplicity of solutions to
BVP(\ref{BVP}) and obtain that is has at least $2N$ distinct solutions
assuming some conditions. Our results are based on \cite{AppMathLett} and 
\cite{limmingGao} by extending these to the case of fourth order discrete
equations.

Lets mention, far from being exhaustive, the following recent papers on
discrete BVPs investigated via critical point theory, \cite{agrawal}, \cite%
{caiYu}, \cite{Liu}, \cite{sehlik}, \cite{TianZeng}, \cite{teraz}, \cite%
{zhangcheng}, \cite{nonzero}. These papers employ in the discrete setting
the variational techniques already known for continuous problems of course
with necessary modifications. The tools employed cover the Morse theory,
mountain pass methodology, linking arguments.

For the sake of convinience lets recall some basic facts and definitions
used in this note, \cite{bdeimling}

Let $X$ be a real Banach space and let $T:X\rightarrow 
%TCIMACRO{\U{211d} }%
%BeginExpansion
\mathbb{R}
%EndExpansion
$. For fixed $x,h\in X$, symbol $\delta T(x;h)$ stands for the Gateaux
derivative of $T$ at point $x$ and direction $h$, while $T^{\prime }(x)$
stands for the Frechet or strong derivative of $T$ at point $x$. By $C^{1}(X,%
%TCIMACRO{\U{211d} }%
%BeginExpansion
\mathbb{R}
%EndExpansion
)$ we denote the set of continously (Frechet) differentable functionals on $%
X $.

If $X=X_{1}\times ...\times X_{n}$ where $X_{i}$ are one-dimensional Banach
spaces we say that $T$ has partial derivative at point $x\in X$ on $X_{i}$
if there exists $L_{x}\in L(X,%
%TCIMACRO{\U{211d} }%
%BeginExpansion
\mathbb{R}
%EndExpansion
)$.such that%
\begin{equation*}
\lim_{h\rightarrow \theta }\frac{||T(x+he_{i})-T(x)-L_{x}h||}{||h||}=0
\end{equation*}

Where $e_{i}\in X_{i}$ is the unit vector. The following lemma from \cite%
{Maurin} gives us an useful relation between continuity of partial
derivatives and being $C^{1}$.

\begin{lemma}
\label{Pczastk} Let $T:U\rightarrow 
%TCIMACRO{\U{211d} }%
%BeginExpansion
\mathbb{R}
%EndExpansion
$ where $U\subset X$ is an open set. Then $T$ is continously differentable
on $U$ if and only if functions 
\begin{equation*}
U\backepsilon x\rightarrow J_{X_{i}}^{\prime }(x)\in L(X_{i},%
%TCIMACRO{\U{211d} }%
%BeginExpansion
\mathbb{R}
%EndExpansion
),
\end{equation*}
are continous on $U$ for $i=1,...,n.$
\end{lemma}

We call $x_{0}$ a critical point of $T$ if and only if $\delta T(x_{0},h)=0$
for all $h\in X$. Functional $T$ satisfies the Palais-Smale condition (P.S.
condition for short) if any sequence $(x_{n}\in X:n\in 
%TCIMACRO{\U{2115} }%
%BeginExpansion
\mathbb{N}
%EndExpansion
)$ for which $\{T(x_{n}):n\in 
%TCIMACRO{\U{2115} }%
%BeginExpansion
\mathbb{N}
%EndExpansion
\}$ is bounded and $T^{\prime }(x_{n})\rightarrow \theta $ as $n\rightarrow
\infty $ possesses a convergent subsequence.

We recall theorems that will be used in research of multiplicity of
solution: the mountain pass theorem and Clark's theorem which will be
essential in proving our main result.

\begin{theorem}
\label{MPT}\cite{ajm}Let $T\in C^{1}(X,%
%TCIMACRO{\U{211d} }%
%BeginExpansion
\mathbb{R}
%EndExpansion
)$ satisfy the P.S. condition. Assume that $T(\theta )=0$, $\Omega \subset X$
is an open set containing $\theta $, $x_{1}\notin \Omega $. If 
\begin{equation*}
\max \{T(\theta ),T(x_{1})\}<\inf_{x\in \partial \Omega }T(x)
\end{equation*}
then 
\begin{equation*}
c=\inf_{h\in \Gamma }\max_{t\in \lbrack 0,1]}J(h(t))
\end{equation*}
is the critical value of $J$, where%
\begin{equation*}
\Gamma =\{h:[0,1]\rightarrow E:h\text{ is continous, }h(0)=\theta
,h(1)=x_{1}\}\text{.}
\end{equation*}
\end{theorem}

\begin{theorem}
\label{ClarkTheo}\cite{rabinowirtz} Let $X$ be a real Banach space and let $%
T\in C^{1}(X,%
%TCIMACRO{\U{211d} }%
%BeginExpansion
\mathbb{R}
%EndExpansion
)$ be even, bounded from below and satisfying the P.S condition. Suppose $%
T(\theta )=0$ and that there exist a set $K\subset X$ such that $K$ is
homeomorphic to $S^{n-1}$ ($(n-1)$-dimensional sphere) by an odd map, and $%
\sup K<0$. Then $T$ posseses at least $n$ distinct pairs of critical points.
\end{theorem}

\section{Main results}

Solutions to (\ref{BVP}) are obtained in space 
\begin{equation*}
E=\{y:Z[-1,N+2]\rightarrow 
%TCIMACRO{\U{211d} }%
%BeginExpansion
\mathbb{R}
%EndExpansion
|y(-1)=y(0)=y(N+1)=y(N+2)=0\}
\end{equation*}%
\newline
considered with a norm 
\begin{equation*}
||y||=\sqrt{\sum\limits_{k=1}^{N}y(k)^{2}}
\end{equation*}

\begin{lemma}
\label{MainLemma} For all $y\in E$ let 
\begin{equation*}
J(y)=\sum\limits_{k=1}^{N+2}\frac{p(k)}{2}(\Delta
^{2}y(k-2))^{2}-\sum\limits_{k=1}^{N+1}\frac{q(k)}{2}(\Delta
y(k-1))^{2}+\sum\limits_{k=1}^{N}F(k,y(k))
\end{equation*}%
Where $F(k,s)=\int_{0}^{s}f(k,t)dt$. Then $J\in C^{1}(E,%
%TCIMACRO{\U{211d} }%
%BeginExpansion
\mathbb{R}
%EndExpansion
)$. Point $y_{0}$ is a solution to (\ref{BVP}) if and only if it is a
critical point of $J$.
\end{lemma}

\begin{proof}
We denote by $\varphi :%
%TCIMACRO{\U{211d} }%
%BeginExpansion
\mathbb{R}
%EndExpansion
\rightarrow 
%TCIMACRO{\U{211d} }%
%BeginExpansion
\mathbb{R}
%EndExpansion
$ the function $\varphi (\varepsilon )=J(y+\varepsilon h)$ for $y,h\in E$
and $\varepsilon \in 
%TCIMACRO{\U{211d} }%
%BeginExpansion
\mathbb{R}
%EndExpansion
$. Then%
\begin{equation*}
\begin{array}{l}
\varphi (\varepsilon )=\sum\limits_{k=1}^{N+2}\frac{p(k)}{2}(\Delta
^{2}(y+\varepsilon h)(k-2))^{2}- \\ 
\sum\limits_{k=1}^{N+1}\frac{q(k)}{2}(\Delta (y+\varepsilon
h)(k-1))^{2}+\sum\limits_{k=1}^{N}F(k,(y+\varepsilon h)(k))%
\end{array}%
\end{equation*}%
$\varphi (\varepsilon )$ is differentable therefore $J$ is differentable in
the sense of Gateaux. Moreover, we have%
\begin{equation*}
\begin{array}{l}
\delta J(y;h){\small =\sum\limits_{k=1}^{N+2}p(k)\Delta }^{2}{\small %
y(k-2)\Delta }^{2}{\small h(k-2)-} \\ 
{\small \sum\limits_{k=1}^{N+1}q(k)\Delta y(k-1)\Delta
h(k-1)+\sum\limits_{k=1}^{N}f(k,y(k))h(k)=} \\ 
\sum\limits_{k=1}^{N+2}{\small p(k)\Delta }^{2}{\small y(k-2)\Delta }^{2}%
{\small h(k-2)-\sum\limits_{k=1}^{N+1}q(k)\Delta y(k-1)\Delta
h(k-1)+\sum\limits_{k=1}^{N}f(k,y(k))h(k)=} \\ 
\sum\limits_{k=1}^{N}[{\small \Delta }^{2}({\small p(k)\Delta }^{2}{\small %
y(k-2))}+{\small \Delta (q(k)y(k-1))+f(k,y(k))]h(k)}%
\end{array}%
\end{equation*}%
Since $h$ is arbitrary fixed it follows that $y_{0}\in E$ is a critical
point of $J$ if and only if $y_{0}$ satisfies BVP (\ref{BVP}). To proof that
assume first that $y_{0}$ is a critical point of $J$, i.e. $\delta
J(y_{0};h)=0$ for all $h\in E\backslash \{0\}$ and put 
\begin{equation*}
{\small h(k)=\Delta }^{2}({\small p(k)\Delta }^{2}y_{0}{\small (k-2))}+%
{\small \Delta (q(k)y_{0}(k-1))+f(k,y_{0}(k))}\text{ for }k=1,..,N\text{.}
\end{equation*}%
Then $\delta J(y_{0};h)=\sum\limits_{k=1}^{N}[{\small \Delta }^{2}({\small %
p(k)\Delta }^{2}y_{0}{\small (k-2))}+{\small \Delta
(q(k)y_{0}(k-1))+f(k,y_{0}(k))]}^{2}=0$ and therefore 
\begin{equation*}
{\small \Delta }^{2}({\small p(k)\Delta }^{2}y_{0}{\small (k-2))}+{\small %
\Delta (q(k)y_{0}(k-1))+f(k,y_{0}(k))=0}\text{ for }k=1,..,N\text{.}
\end{equation*}%
On the other side if $y_{0}$ is solution to the above system then $\delta
J(y_{0};h)=0$ for every $h\in E\backslash \{0\}$.\newline
Since $f$ is continous on its second variable it follows that $\delta J(y,h)$
is continous. It holds that $J_{X_{i}}^{\prime }(y)=\delta J(y,e_{i})$,
where $X_{i}=\{y\in E:y(k)=0$ for $k\neq i\}$ and $e_{i}\in X_{i}$ are unit
vectors for $i=1,..,N$. Thus all partial derivates are continous. Then from
lemma(\ref{Pczastk}) it follows that $J\in C^{1}(E,%
%TCIMACRO{\U{211d} }%
%BeginExpansion
\mathbb{R}
%EndExpansion
)$.
\end{proof}

\begin{lemma}
\label{nier4} For all $y\in E$ it holds that 
\begin{equation*}
\sum\limits_{k=1}^{N+1}(\Delta y(k-1))^{2}\leq 4||y||^{2}\text{ and }%
\sum\limits_{k=1}^{N+2}(\Delta ^{2}y(k-2))^{2}\leq 16||y||^{2}
\end{equation*}
\end{lemma}

\begin{proof}
Since $-2ab\leq a^{2}+b^{2}$ for all $a,b\in 
%TCIMACRO{\U{211d} }%
%BeginExpansion
\mathbb{R}
%EndExpansion
$ we have 
\begin{equation*}
\begin{array}{l}
\sum\limits_{k=1}^{N+1}{\small (\Delta y(k-1))}^{2}{\small =}%
\sum\limits_{k=1}^{N+1}{\small (y(k)-y(k-1))}^{2}{\small =} \\ 
\sum\limits_{k=1}^{N+1}{\small (y(k)}^{2}{\small -2y(k)y(k-1)+y(k-1)}^{2}%
{\small )\leq }\sum\limits_{k=1}^{N+1}{\small (2y(k)}^{2}{\small +2y(k-1)}%
^{2}{\small )=}\sum\limits_{k=1}^{N}{\small 2y(k)}^{2}{\small %
+\sum\limits_{k=1}^{N}2y(k)}^{2}\leq 4||y||^{2}%
\end{array}%
\end{equation*}%
and 
\begin{equation*}
\begin{array}{l}
\sum\limits_{k=1}^{N+2}{\small (\Delta }^{2}{\small y(k-2))}^{2}{\small =}%
\sum\limits_{k=1}^{N+2}{\small (\Delta y(k-1)-\Delta y(k-2))}^{2} \\ 
{\small \leq 4\sum\limits_{k=1}^{N+1}(\Delta y(k-1))^{2}\leq 16||y||^{2}.}%
\end{array}%
\end{equation*}%
\newline
\end{proof}

\begin{lemma}
\label{lemmamat} Let $A_{m\times m}$ be a symmetric and positive-defined
real matrix and let $B_{m\times n}$ be a real matrix \ Then $B^{T}AB$ is
positive defined if and only if $Rank(B)=n$

\begin{proof}
If $B^{T}AB$ is positive defined for all $x\in 
%TCIMACRO{\U{211d} }%
%BeginExpansion
\mathbb{R}
%EndExpansion
^{n}\backslash \{\theta \}$ we have 
\begin{equation*}
(Bx)^{T}A(Bx)=x^{T}B^{T}ABx>0,
\end{equation*}%
hence $Bx\neq \theta $, and therefore $Rank(B)=n$. Assume that $Rank(B)=n$.
Then for all $x\in 
%TCIMACRO{\U{211d} }%
%BeginExpansion
\mathbb{R}
%EndExpansion
^{n}\backslash \{\theta \}$ it holds that $Bx\neq \theta $ and .$%
(Bx)^{T}A(Bx)>0$ since $A$ is positive defined, hence $B^{T}AB$ is positive
defined.
\end{proof}
\end{lemma}

%\newline
Define $v(k)=\Delta y(k)=y(k+1)-y(k)$ for $k\in Z[0,N]$. Then $v=V\widetilde{%
y}$, where 
\begin{equation*}
v=[v(0),v(1),...,v(N)]^{T},\widetilde{y}=[y(1),..,y(N)]^{T},
\end{equation*}%
and 
\begin{equation*}
{\small V=}\left[ 
\begin{array}{ccccc}
1 &  &  &  &  \\ 
-1 & 1 &  &  &  \\ 
& -1 & \ddots  &  &  \\ 
&  & \ddots  &  &  \\ 
&  &  &  & 1 \\ 
&  &  &  & -1%
\end{array}%
\right] _{(N+1)\times N}
\end{equation*}%
Note that 
\begin{equation*}
\sum\limits_{k=1}^{N+1}(\Delta y(k-1))^{2}=v^{T}v=(V\widetilde{y})^{T}(V%
\widetilde{y})=\widetilde{y}^{T}V^{T}V\widetilde{y}
\end{equation*}%
By lemma (\ref{lemmamat}) with $A=I_{(N+1)\times (N+1)}$ it follows that $%
V^{T}V$ is positive-defined. Therefore all eigenvalues of\ $V^{T}V$ are real
and positive. Denote henceforth by $\lambda _{1}$ the smallest eigenvalue of 
$V^{T}V$. Then it follows that for all $y\in E$. 
\begin{equation*}
\sum\limits_{k=1}^{N+1}(\Delta y(k-1))^{2}=\widetilde{y}^{T}V^{T}V%
\widetilde{y}\geq \lambda _{1}\widetilde{y}^{T}\widetilde{y}=\lambda
_{1}||y||^{2}
\end{equation*}%
In the same way we put $w(k)=\Delta ^{2}y(k)$ for $k\in Z[-1,N]$. Then $w=W%
\widetilde{y}$, where $w=[w(-1),w(0),...,w(N)]^{T}$ and 
\begin{equation*}
{\small W=}\left[ 
\begin{array}{ccccc}
1 &  &  &  &  \\ 
-2 & 1 &  &  &  \\ 
1 & -2 & \ddots  &  &  \\ 
& 1 & \ddots  & 1 &  \\ 
&  & \ddots  & -2 & 1 \\ 
&  &  & 1 & -2 \\ 
&  &  &  & 1%
\end{array}%
\right] _{(N+2)\times N}
\end{equation*}%
Likewise in previous case by lemma (\ref{lemmamat}) $W^{T}W$\ is positive
defined.\ Denote henceforth by $\lambda _{2}$ the smallest eigenvalue of $%
W^{T}W$\ \ Then for all $y\in E$ \ it holds that%
\begin{equation*}
\sum\limits_{k=1}^{N+2}(\Delta ^{2}y(k-2))^{2}=\widetilde{y}^{T}W^{T}W%
\widetilde{y}\geq \lambda _{2}||y||^{2}
\end{equation*}%
So we have proven the following lemma:

\begin{lemma}
\label{lemmalamb}For all $y\in E$ it follows that \ 
\begin{equation*}
\sum\limits_{k=1}^{N+1}(\Delta y(k-1))^{2}\geq \lambda _{1}||y||^{2}\text{
and }\sum\limits_{k=1}^{N+2}(\Delta ^{2}y(k-2))^{2}\geq \lambda
_{2}|||y||^{2}
\end{equation*}
\end{lemma}

Define $p_{\min }=\min_{k\in Z[1,N+2]}\{p(k)\}$, $p_{\max }=\max_{k\in
Z[1,N+2]}\{p(k)\}$ and $q_{\min }$, $q_{\max }$ in the same manner. and let
for $p:Z[1,N+2]\rightarrow 
%TCIMACRO{\U{211d} }%
%BeginExpansion
\mathbb{R}
%EndExpansion
$ and $q:Z[1,N+1]\rightarrow 
%TCIMACRO{\U{211d} }%
%BeginExpansion
\mathbb{R}
%EndExpansion
$: 
\begin{equation*}
\eta ^{\prime }(p)=\left\{ 
\begin{array}{cc}
\lambda _{2} & \text{if }p_{\min }\geq 0 \\ 
16 & \text{if }p_{\min }<0%
\end{array}%
\right. \text{,}
\end{equation*}%
\begin{equation*}
\eta (q)=\left\{ 
\begin{array}{cc}
\lambda _{1} & \text{if }q_{\max }<0 \\ 
4 & \text{if }q_{\max }\geq 0%
\end{array}%
\right. \text{.}
\end{equation*}

\begin{theorem}
\label{1sol} Assume that\newline
1) there exist $m>0$ such that $sf(k,s)\geq 0$ for $|s|\geq m$ and $k\in
Z[1,N]$ \newline
2) $\eta ^{\prime }(p)p_{\min }-\eta (q)q_{\max }>0$\newline
Than $J$ coercive and BVP\ (\ref{BVP}) has at least one solution. Moreover
if there exist $k_{0}\in Z[1,N]$ such that $f(k_{0},0)\neq 0$ the solution
is non-zero.
\end{theorem}

\begin{proof}
For all $k\in Z[1,N]$ it follows that 
\begin{equation*}
\int\limits_{0}^{s}f(k,t)dt\geq \int\limits_{0}^{s}-sgn(s)|f(k,t)|dt\geq
\int\limits_{-m}^{m}-sgn(s)|f(k,t)|dt>-\infty
\end{equation*}
since $f$ is continous on second variable and $[-m,m]$ bounded.\newline
Put $C_{1}=\sum\limits_{k=1}^{N}\int\limits_{m}^{-m}-|f(k,t)|dt$. We have%
\begin{equation*}
{\small \bigskip }%
\begin{array}{l}
{\small J(y)}{\small =}\sum\limits_{k=1}^{N+2}{\small (}\frac{p(k)}{2}%
{\small (\Delta }^{2}{\small y(k-2))}^{2}{\small -}\sum\limits_{k=1}^{N+1}%
\frac{q(k)}{2}{\small (\Delta y(k-1))}^{2}{\small +\sum%
\limits_{k=1}^{N}F(k,y(k))\geq \bigskip } \\ 
\sum\limits_{k=1}^{N+2}\frac{p_{\min }}{2}{\small (\Delta }^{2}{\small %
y(k-2))}^{2}{\small -}\sum\limits_{k=1}^{N+1}\frac{q_{\max }}{2}{\small %
(\Delta y(k-1))}^{2}{\small +}C_{1}\geq \frac{1}{2}{\small (}\eta ^{\prime
}(p)p_{\min }{\small -\eta (q)q}_{\max })||y||^{2}+C_{1}\text{.}%
\end{array}%
\end{equation*}
\end{proof}

And from 2) it follows that $J$ is coercive, and since it is coercive and $%
C^{1}$ there exist critical point $y_{0}\in E$ such that $%
J(y_{0})=\min_{y\in E}J(y)$. Therefore BVP (\ref{BVP}) has a solution. If $%
f(k_{0},0)\neq 0$, for $y=\theta $ we have

\begin{equation*}
\Delta ^{2}(p(k_{0})\Delta ^{2}y(k_{0}-2))+\Delta (q(k_{0})\Delta
y(k_{0}-1))+f(k_{0},y(k_{0}))=f(k_{0},0)\neq 0\newline
\end{equation*}

hence $\theta $ is not a solution to BVP (\ref{BVP})

\begin{theorem}
Assume that 1), 2) of theorem (\ref{1sol}) hold, $f$ is non-decreasing on $s$
for $k\in \lbrack 1,N]$ and \textit{\ }$p_{\min }\eta ^{\prime }(p)>q_{\max
}\eta (q)$ Then BVP\ (\ref{BVP}) has exactly one solution.
\end{theorem}

\begin{proof}
By assumptions 1) and 2) of theorem (\ref{1sol}) there exists at least one
solution to BVP (\ref{BVP}). Since $f$ is non-decreasing on $F$ is convex on 
$s$ for $k\in \lbrack 1,N]$, and therefore $y\rightarrow $ $%
\sum\limits_{k=1}^{N}F(k,y(k))$ is convex. Let for $y\in E$, 
\begin{equation*}
I(y)=\sum\limits_{k=1}^{N+2}{\small (}\frac{p(k)}{2}{\small (\Delta }^{2}%
{\small y(k-2))}^{2}{\small -}\sum\limits_{k=1}^{N+1}\frac{q(k)}{2}{\small %
(\Delta y(k-1))}^{2}.
\end{equation*}%
For arbitrary fixed $y,h\in E$, $h\neq 0$ we have

\begin{equation*}
\begin{array}{l}
I(y+h)-I(y)=\sum\limits_{k=1}^{N+2}\frac{p(k)}{2}[2\Delta ^{2}h(k-2)\Delta
^{2}y(k-2)+\Delta ^{2}h(k-2)^{2}]- \\ 
\sum\limits_{k=1}^{N+1}\frac{q(k)}{2}[2\Delta h(k-1)\Delta y(k-1)+\Delta
h(k-1)^{2}]-\sum\limits_{k=1}^{N}[{\small \Delta }^{2}({\small p(k)\Delta }%
^{2}{\small y(k-2))}+{\small \Delta (q(k)y(k-1))}{\small =} \\ 
\sum\limits_{k=1}^{N+2}p(k)\Delta ^{2}y(k-2)\Delta
^{2}h(k-2)+\sum\limits_{k=1}^{N+2}\frac{1}{2}p(k)\Delta ^{2}h(k-2)^{2}- \\ 
\sum\limits_{k=1}^{N+1}q(k)\Delta y(k-1)\Delta
h(k-1)-\sum\limits_{k=1}^{N+1}\frac{1}{2}q(k)\Delta h(k-1)^{2}%
\end{array}%
\end{equation*}

and

\begin{equation*}
\begin{array}{l}
I(y+h)-I(y)-\delta I(y;h)= \\ 
\sum\limits_{k=1}^{N+2}p(k)\Delta ^{2}y(k-2)\Delta
^{2}h(k-2)+\sum\limits_{k=1}^{N+2}\frac{1}{2}p(k)\Delta
^{2}h(k-2)^{2}-\sum\limits_{k=1}^{N+1}q(k)\Delta y(k-1)\Delta h(k-1)- \\ 
\sum\limits_{k=1}^{N+1}\frac{1}{2}q(k)\Delta
h(k-1)^{2}-\sum\limits_{k=1}^{N+2}{\small p(k)\Delta }^{2}{\small %
y(k-2)\Delta }^{2}{\small h(k-2)+\sum\limits_{k=1}^{N+1}q(k)\Delta
y(k-1)\Delta h(k-1)=} \\ 
\sum\limits_{k=1}^{N+2}\frac{1}{2}p(k)\Delta
^{2}h(k-2)^{2}-\sum\limits_{k=1}^{N+1}\frac{1}{2}q(k)\Delta h(k-1)^{2}\geq
\\ 
\frac{1}{2}\left[ \sum\limits_{k=1}^{N+2}p_{\min }\Delta
^{2}h(k-2)^{2}-\sum\limits_{k=1}^{N+1}q_{\max }\Delta h(k-1)^{2}\right]
\geq \frac{1}{2}(p_{\min }\eta ^{\prime }(p)-q_{\max }\eta (q))||h||^{2}>0%
\end{array}%
\end{equation*}%
hence $I$ is strongly convex. Therefore $J$ is sum of convex and strongly
convex functionals, hence it is strongly convex. That implies that the
solution to (\ref{BVP}) is unique.
\end{proof}

Put 
\begin{equation*}
\alpha _{1}=\eta (q){\small q}_{\max }{\small -\eta ^{\prime }(p)p}_{\min }
\end{equation*}

\begin{lemma}
\label{lem1} Assume that \newline
1) $\min_{k\in Z[1,N]}\lim_{s\rightarrow +\infty }\frac{f(k,s)}{s}{\small >}%
\alpha _{1}$ \newline
2) there exists $S>0$ such that $f(k,-s)\leq -f(k,s)$. \newline
Then $J$ is coercive and BVP\ (\ref{BVP}) has at least one solution.
\end{lemma}

\begin{proof}
For\ every $s<-S$ and fixed $k$ we have:\newline
\begin{eqnarray*}
f(k,-s) &\leq &-f(k,s) \\
\frac{f(k,-s)}{-s} &\geq &\frac{-f(k,s)}{-s} \\
\frac{f(k,-s)}{-s} &\geq &\frac{f(k,s)}{s} \\
\lim_{s\rightarrow -\infty }\frac{f(k,s)}{s} &\geq &\lim_{s\rightarrow
\infty }\frac{f(k,s)}{s}{\small >}\alpha _{1}\text{.}
\end{eqnarray*}%
Therefore $\min_{k\in Z[1,N]}\lim_{|s|\rightarrow +\infty }\frac{f(k,s)}{s}%
{\small >}\alpha _{1}$. Take a number $\beta $ such that 
\begin{equation*}
\min_{k\in Z[1,N]}\lim_{|s|\rightarrow \infty }\frac{f(k,s)}{s}\geq \beta 
{\small >}\alpha _{1}\text{.}
\end{equation*}%
for $\varepsilon =\frac{1}{2}(\beta -\alpha _{1})>0$, there exist a constant 
$r>0$%
\begin{eqnarray*}
\frac{f(k,s)}{s} &\geq &\beta -\varepsilon \text{ for }|s|\geq M\text{, }%
k\in Z[1,N]\text{,} \\
F(k,s) &\geq &\frac{\beta -\varepsilon }{2}s^{2}\text{ for }|s|\geq M\text{, 
}k\in Z[1,N]
\end{eqnarray*}%
Therefore, there exists a constant $C$ such that 
\begin{equation*}
F(k,s)\geq \frac{\beta -\varepsilon }{2}s^{2}+C\text{ for }s\in 
%TCIMACRO{\U{211d} }%
%BeginExpansion
\mathbb{R}
%EndExpansion
\text{, }k\in Z[1,N]
\end{equation*}%
Using this inequality we obtain the following:%
\begin{equation*}
\begin{array}{l}
{\small J(y)=}\sum\limits_{k=1}^{N+2}{\small (}\frac{p(k)}{2}{\small %
(\Delta }^{2}{\small y(k-2))}^{2}{\small -}\sum\limits_{k=1}^{N+1}\frac{q(k)%
}{2}{\small (\Delta y(k-1))}^{2}{\small +\sum\limits_{k=1}^{N}F(k,y(k))\geq
\bigskip } \\ 
\sum\limits_{k=1}^{N+2}\frac{p_{\min }}{2}{\small (\Delta }^{2}{\small %
y(k-2))}^{2}{\small -}\sum\limits_{k=1}^{N+1}\frac{q_{\max }}{2}{\small %
(\Delta y(k-1))}^{2}{\small +\sum\limits_{k=1}^{N}}\frac{\beta -\varepsilon 
}{2}{\small y(k)}^{2}{\small +NC\geq \bigskip } \\ 
\frac{1}{2}{\small (}\eta ^{\prime }(p)p_{\min }{\small -\eta (q)q}_{\max }%
{\small +\beta -\varepsilon )||y||}^{2}{\small +NC\geq }\frac{1}{2}{\small (}%
\eta ^{\prime }(p)p_{\min }{\small -\eta (q)q}_{\max }{\small +\alpha }%
_{1}+\varepsilon {\small )||y||}^{2}{\small +NC}%
\end{array}%
\end{equation*}%
Then it follows from definition of $\alpha _{1}$ that $\eta ^{\prime
}(p)p_{\min }{\small -\eta (q)q}_{\max }{\small +\alpha }_{1}+\varepsilon
=\varepsilon >0$ and therefore $J$ is coercive. Since $J$ is coercive and $%
C^{1}$ there exist $x^{\ast }$ such that $J(x^{\ast })=\inf_{x\in E}J(x)$
and $J^{\prime }(x^{\ast })=\theta $.and therefore BVP (\ref{BVP}) has at
least one solution $x^{\ast }$.
\end{proof}

Put%
\begin{equation*}
\xi (q)=\left\{ 
\begin{array}{cc}
\lambda _{1} & q_{\min }\geq 0 \\ 
4 & q_{\min }<0%
\end{array}%
\right. \text{ for }Z[1,N+1]\rightarrow 
%TCIMACRO{\U{211d} }%
%BeginExpansion
\mathbb{R}
%EndExpansion
\text{.}
\end{equation*}%
\newline
and%
\begin{equation*}
\alpha _{2}={\small \xi (q)}q_{\min }-{\small 16}p_{\max }
\end{equation*}

\begin{theorem}
\label{MPT1} Assume that $\dim E>1$, $p_{\max }>0$,\ 1) and 2) of lemma (\ref%
{lem1}) hold and that 
\begin{equation*}
\max_{k\in Z[1,N]}\lim_{s\rightarrow 0}\frac{f(k,s)}{s}<\alpha _{2}\text{.}
\end{equation*}%
Then BVP (\ref{BVP}) has at least two solutions.

\begin{proof}
Note that $J(\theta )=0$. By 1) and 2) of lemma (\ref{lem1}) and $J$ is
coercive and since it is also continous it satisfies P.S. condition.
Therefore there exists a critical point $y_{1}$.\newline
Take a number $\beta $ such that 
\begin{equation*}
\max_{k\in Z[1,N]}\lim_{s\rightarrow 0}\frac{f(k,s)}{s}\leq \beta <\alpha
_{2}\text{.}
\end{equation*}%
For $\varepsilon =\frac{1}{2}(\alpha _{2}-\beta )>0$, there exist a constant 
$\delta >0$ such that 
\begin{equation*}
\frac{f(k,s)}{s}\leq \beta +\varepsilon \text{, for all }|s|\leq \delta 
\text{ and }k\in Z[1,N]
\end{equation*}%
\newline
there exist $\delta >0$ such that%
\begin{eqnarray*}
\frac{f(k,s)}{s} &\leq &\beta +\varepsilon \\
F(k,s) &\leq &\frac{1}{2}(\beta +\varepsilon )s^{2}
\end{eqnarray*}%
for all $|s|\leq \delta $ and $k\in Z[1,N]$. Let 
\begin{equation*}
\Omega =\{y\in E:||y||<\delta \}\text{.}
\end{equation*}%
Then $\partial \Omega =\{y\in E:||y||=\delta \}$. Note that for all $y\in
\partial \Omega $ and $k\in Z[1,N]$ $|y(k)|\leq \delta $. and%
\begin{equation*}
\begin{array}{l}
{\small J(y)}{\small =}\sum\limits_{k=1}^{N+2}{\small (}\frac{p(k)}{2}%
{\small (\Delta }^{2}{\small y(k-2))}^{2}{\small -}\sum\limits_{k=1}^{N+1}%
\frac{q(k)}{2}{\small (\Delta y(k-1))}^{2}{\small +\sum%
\limits_{k=1}^{N}F(k,y(k))} \\ 
{\small \leq }\frac{1}{2}{\small (16}p_{\max }{\small -}q_{\min }{\small \xi
(q)+}\beta +\varepsilon )||y||^{2}=\frac{1}{2}{\small (16}p_{\max }{\small -}%
q_{\min }{\small \xi (q)+}\alpha _{2}-\varepsilon )||y||^{2}<0\text{.}%
\end{array}%
\end{equation*}%
\newline
and since $\partial \Omega $ is compact 
\begin{equation*}
\sup \{J(y):y\in \partial \Omega \}<0.
\end{equation*}%
Since $J$ is coercive, there exist $y_{1}$ with $||y_{1}||>\delta $ such
that $\min \{0,J(y_{1})\}=0$. Therefore 
\begin{equation*}
\min \{0,J(y_{1})\}>\sup \{J(y):y\in \partial \Omega \}\text{.}
\end{equation*}%
and $-J$ satisfies all assumptions or theorem (\ref{MPT}) what implies that
there exist $y_{2}\in E$ such that $J^{\prime }(y_{2})=\theta $, and 
\begin{equation*}
J(y_{2})=\sup_{h\in \Gamma }\min_{t\in \lbrack 0,1]}J(h(t))\text{.}
\end{equation*}%
Assume that $y_{1}=y_{2}.$That is 
\begin{equation*}
J(y_{2})=\sup_{h\in \Gamma }\min_{t\in \lbrack 0,1]}J(h(t))=\inf_{y\in
E}J(y)=J(y_{1}).
\end{equation*}%
Therefore for all $h\in \Gamma $ 
\begin{equation*}
\min_{t\in \lbrack 0,1]}J(h(t))\leq \inf_{y\in E}J(y)\text{.}
\end{equation*}%
and since $h([0,1])\subseteq E$ 
\begin{equation*}
\min_{t\in \lbrack 0,1]}J(h(t))\geq \inf_{y\in E}J(y)
\end{equation*}%
Therefore 
\begin{equation*}
\min_{t\in \lbrack 0,1]}J(h(t))=\inf_{y\in E}J(y)
\end{equation*}%
Since $\dim E>1$ there exist $h_{1},h_{2}\in \Gamma $ such that $%
h_{1}[(0,1)]\cap h_{2}[(0,1)]=\varnothing $. From the continuity of $h_{1}$
and $h_{2}$ there exist $t_{1},t_{2}\in (0,1)$ such that $%
h_{1}(t_{1})=\min_{t\in \lbrack 0,1]}J(h_{1}(t_{1}))$, $h_{2}(t_{2})=\min_{t%
\in \lbrack 0,1]}J(h_{2}(t_{2}))$.Therefore $h_{1}(t_{1})$ and $h_{2}(t_{2})$
are two different critical points.i.e. they are two different solutions to
BVP\ (\ref{BVP})\newline
\end{proof}
\end{theorem}

\begin{theorem}
\label{theo1} Assume that $f$ is odd on its second variable for all $k\in
Z[1,N]$ and that\newline
1) $p_{\min }\geq 0$\newline
2) $\min_{k\in Z[1,N]}\lim_{s\rightarrow \infty }\frac{f(k,s)}{s}>\alpha
_{1} $\newline
3) $\max_{k\in Z[1,N]}\lim_{s\rightarrow 0^{+}}\frac{f(k,s)}{s}<\alpha _{2}$%
\newline
Then BVP (\ref{BVP}) has at least $2N$ distinct solutions.
\end{theorem}

\begin{proof}
Since $f$ is odd, assumption 2) of lemma (\ref{lem1}) is satisfied. From
this and 2) by lemma (\ref{lem1}) it follows that $J$ is coercive (and
therefore bounded from below on $E$ and satisfying the P.S. condition).
Since $f$ is odd on its second variable, $J$ is even and $\frac{f(k,-s)}{-s}=%
\frac{f(k,s)}{s}$ for $s\in 
%TCIMACRO{\U{211d} }%
%BeginExpansion
\mathbb{R}
%EndExpansion
$ and $k=1,...,N$. Therefore there exist limits $\lim_{s\rightarrow 0}\frac{%
f(k,s)}{s}$, $k=1,...N$. Take a number $\beta $ such that%
\begin{equation*}
\max_{k\in Z[1,N]}\lim_{s\rightarrow 0}\frac{f(k,s)}{s}\leq \beta <\alpha
_{2}.
\end{equation*}%
For $\varepsilon =\frac{1}{2}(\alpha _{2}-\beta )>0$, there exist a constant 
$\delta >0$ such that 
\begin{equation*}
\begin{array}{l}
\frac{f(k,s)}{s}\leq \beta +\varepsilon \\ 
f(k,s)\leq \frac{\beta +\varepsilon }{2}s \\ 
F(k,s)\leq \frac{\beta +\varepsilon }{2}s^{2}%
\end{array}%
\end{equation*}%
for all $|s|\leq \delta $ and $k\in Z[1,N]$\newline
It follows that for all $y\in K$ and $k\in Z[1,N]$, $F(k,y(k))\leq \frac{%
\beta +\varepsilon }{2}y(k)^{2}$. Let $0<\delta ^{\prime }<\delta $, $%
K=\{y\in E:||y||=\delta ^{\prime }\}$. Then $K$ is homeomorphic to a $%
S^{N-1} $ by an odd map. For $k\in Z[1,N]$ and $y\in K$ we have $|y(k)|\leq
||y||=\delta ^{\prime }<\delta $. Therefore for $y\in K$.

\begin{equation*}
\begin{array}{l}
{\small J(y)}{\small =}\sum\limits_{k=1}^{N+2}\frac{p(k)}{2}{\small (\Delta 
}^{2}{\small y(k-2))}^{2}{\small -}\sum\limits_{k=1}^{N+1}\frac{q(k)}{2}%
{\small (\Delta y(k-1))}^{2}{\small +\sum\limits_{k=1}^{N}F(k,y(k))\leq }%
\frac{1}{2}{\small \cdot (16}p_{\max }{\small -\xi (q)}q_{\min }{\small %
+\beta +\varepsilon ))||y||}^{2} \\ 
{\small \leq }\frac{1}{2}{\small \cdot (16}p_{\max }{\small -\xi (q)}q_{\min
}{\small +\beta +\varepsilon ))\delta }^{\prime }{}^{2}{\small <}\frac{1}{2}%
{\small \cdot (16}p_{\max }{\small -\xi (q)}q_{\min }{\small +\alpha
_{2}-\varepsilon )\delta }^{\prime 2}{\small <0}\text{,}%
\end{array}%
\end{equation*}%
Since $K$ is compact $\sup K=\max K<0$. Therefore all assumptions of theorem
(\ref{ClarkTheo}) are satisfied and it follows that $J$ has at least $N$
distinct pairs of critical points, that is BVP (\ref{BVP}) has at least N
distinct pairs of solutions. \newline
Put 
\begin{equation*}
\alpha _{3}=\min \{{\small \lambda }_{1}{\small q}_{\min }{\small -\lambda }%
_{2}{\small p}_{\max }{\small ,4q}_{\min }{\small -\lambda }_{2}{\small p}%
_{\max }\}
\end{equation*}
\end{proof}

\begin{theorem}
\label{theo2}Assume that $f$ is odd on its second variable for all $k\in
Z[1,N]$ and that\newline
1) $p_{\max }\leq 0$ \newline
2) $\min_{k\in Z[1,N]}\lim_{s\rightarrow \infty }\frac{{\small f(k,s)}}{%
{\small s}}{\small >}\alpha _{2}$\newline
3) $\max_{k\in Z[1,N]}\lim_{s\rightarrow 0}\frac{{\small f(k,s)}}{{\small s}}%
{\small <}\alpha _{3}$\newline
Then BVP (\ref{BVP})\ has at least $2N$ distinct solutions.
\end{theorem}

\begin{proof}
Since $f$ is odd and 2) holds, by lemma (\ref{lem1}) $J$ is even, bounded
from below and satisfies P.S. condition. Take a number $\beta $ such that $%
\lim_{s\rightarrow 0}\frac{f(k,s)}{s}\leq \beta $.\newline
For $\varepsilon =\frac{1}{2}(\alpha _{3}-\beta )>0$ there exist a constant $%
\delta >0$ such that%
\begin{equation*}
\frac{f(k,s)}{s}\leq \beta +\varepsilon \text{ for all }|s|\leq \delta ,%
\text{ and }k\in Z[1,N].
\end{equation*}
\end{proof}

As in proof of theorem for fixed $\delta ^{\prime }\in (0,\delta )$ and each 
$y\in K=\{y\in E:||y||=\delta ^{\prime }\}$ we have

\begin{equation*}
\begin{array}{l}
{\small J(y)=}\sum\limits_{k=1}^{N+2}{\small (}\frac{p(k)}{2}{\small %
(\Delta }^{2}{\small y(k-2))}^{2}{\small -}\sum\limits_{k=1}^{N+1}\frac{q(k)%
}{2}{\small (\Delta y(k-1))}^{2}{\small +\sum\limits_{k=1}^{N}F(k,y(k))\leq 
} \\ 
\sum\limits_{k=1}^{N+2}\frac{p_{\max }}{2}{\small (\Delta }^{2}{\small %
y(k-2))}^{2}{\small -}\sum\limits_{k=1}^{N+1}\frac{q_{\min }}{2}{\small %
(\Delta y(k-1))}^{2}{\small +}\frac{\beta +\varepsilon }{2}%
\sum\limits_{k=1}^{N}{\small y(k)}^{2} \\ 
{\small \leq }\frac{1}{2}{\small (\lambda }_{2}p_{\max }{\small -\xi (q)}%
q_{\min }{\small +\beta +\varepsilon )\delta }^{\prime 2}{\small <}\frac{1}{2%
}{\small (\lambda }_{2}p_{\max }{\small -\xi (q)}q_{\min }{\small +\alpha }%
_{3}{\small )\delta }^{\prime }{}^{2}%
\end{array}%
\end{equation*}%
\newline
Therefore all assumptions of Clark's theorem are satisfied and $J$ has at
least $N$ distinct pairs of critical points, that is BVP(\ref{BVP}) has at
least N distinct pairs of solutions.

\begin{example}
Let us consider the following BVP:
\end{example}

\begin{equation*}
\left\{ 
\begin{array}{c}
\Delta ^{4}y(k-2)-\Delta ^{2}y(k-1)+\frac{1}{3}y(k)^{3}+y(k)=0\text{ for }%
k\in Z[1,N] \\ 
y(-1)=y(0)=y(N+1)=y(N+2)=0%
\end{array}%
\right.
\end{equation*}

Here $f(k,s)=\frac{1}{3}s^{3}+s$ for $k\in Z[1,N]$, $p=q=1$ Then $f(k,s)$ is
odd and$\ \lim_{s\rightarrow \infty }\frac{\frac{1}{3}s^{3}+s}{s}{\small %
=+\infty }$ and $\lim_{s\rightarrow 0}\frac{\frac{1}{3}s^{3}+s}{s}=1<12=\min
\{12,16-\lambda _{1}\}$. Assumptions of theorem (\ref{theo1}) are satisfied
therefore this BVP has at least $2N$ solutions.


\begin{thebibliography}{99}
\bibitem{elyadi} S. N. Elaydi, \textit{An Introduction to Difference
Equations}, Undergrad. Texts Math., Springer-Verlag, New York, 1999.

\bibitem{bdeimling} K. Deimling, Nonlinear Functional Analysis,
Springler-Verlag, Berlin 1985.

\bibitem{limmingGao} L. Gao, Existence of multiple solutions for a
second-order difference equation with a parameter, \textit{Appl. Math.
Comput.} \textbf{216} (2010), no. 5, 1592-1598.

\bibitem{AppMathLett} F. Lian, Y. Xu, Multiple solutions for boundary value
problems of a discrete generalized Emden--Fowler equation, \textit{Appl.
Math. Lett.} \textbf{23} (2010), no. 2, 8-12.cond-order ODEs with Dirichlet
boundary data and variable parameters, \textit{Ill. J. Math.} \textbf{47}
(2003), no. 4, 1189-1206.

\bibitem{ajm} J. Mawhin, Probl\`{e}mes de Dirichlet variationnels non lin%
\'{e}aires, Les Presses de l'Universit\'{e} de Montr\'{e}al, 1987.

\bibitem{rabinowirtz} P.H. Rabinowitz, \textit{Minimax Methods in Critical
Point Theory with Applications to Differential Equations}, CBMS, Reg. Conf.
Ser. Math., vol. 65, Amer. Math. Soc., Providence, RI, 1986.

\bibitem{agrawal} R.P. Agarwal, K. Perera, D. O'Regan, Multiple positive
solutions of singular discrete p-Laplacian problems via variational methods,
Adv. Difference Equ. 2005 (2) (2005) 93--99.

\bibitem{caiYu} X. Cai, J. Yu, Existence theorems of periodic solutions for
second-order nonlinear difference equations, Adv. Difference Equ. 2008
(2008) Article ID 247071.

\bibitem{Liu} J.Q. Liu, J.B. Su, Remarks on multiple nontrivial solutions
for quasi-linear resonant problemes, J. Math. Anal. Appl. 258 (2001)
209--222.

\bibitem{sehlik} P. Stehl\'{\i}k, On variational methods for periodic
discrete problems, J. Difference Equ. Appl. 14 (3) (2008) 259--273.

\bibitem{TianZeng} Y. Tian, Z. Du, W. Ge, Existence results for discrete
Sturm-Liouville problem via variational methods, J. Difference Equ. Appl. 13
(6) (2007) 467--478.

\bibitem{teraz} Y. Yang and J. Zhang, Existence of solution for some
discrete value problems with a parameter, Appl. Math. Comput. 211 (2009),
293--302.

\bibitem{zhangcheng} G. Zhang, S.S. Cheng, Existence of solutions for a
nonlinear system with a parameter, J. Math. Anal. Appl. 314 (1) (2006)
311--319.

\bibitem{nonzero} G. Zhang, Existence of non-zero solutions for a nonlinear
system with a parameter, Nonlinear Anal. 66 (6) (2007) 1400--1416.

\bibitem{Maurin} Maurin, Krzysztof, Analysis. Part I: Elements. Translated
from the original Polish by Eugene Lepa. (English),
\end{thebibliography}
\end{document}